\documentclass[12pt]{amsart}

\usepackage{amsmath,amssymb,amsfonts,amsthm,latexsym,graphicx,multirow}
\usepackage{hyperref}

\def\A{\mathrm{A}} \def\AGL{\mathrm{AGL}}  
 
 \def\calB{\mathcal{B}}

\def\G{\mathrm{G}}   \def\GL{\mathrm{GL}}  \def\GU{\mathrm{GU}}

\def\M{\mathrm{M}} \def\magma{{\sc Magma} }
\def\N{\mathrm{N}}

\def\Out{\mathrm{Out}}
     \def\POm{\mathrm{P\Omega}} \def\PSL{\mathrm{PSL}}    \def\PSp{\mathrm{PSp}} \def\PSU{\mathrm{PSU}}
\def\Q{\mathrm{Q}}
   
\def\SL{\mathrm{SL}} \def\SO{\mathrm{SO}}  \def\Sp{\mathrm{Sp}}  \def\SU{\mathrm{SU}}  \def\Sy{\mathrm{S}}  \def\Sz{\mathrm{Sz}}
\def\Z{\mathrm{C}} 

\newtheorem{theorem}{Theorem}[section]
\newtheorem{lemma}[theorem]{Lemma}

\newtheorem{corollary}[theorem]{Corollary}

\theoremstyle{definition}

\newtheorem*{remark}{Remark}
\newtheorem{hypothesis}[theorem]{Hypothesis}

\begin{document}

\title{Quasiprimitive groups containing a transitive alternating group}

\author[Xia]{Binzhou Xia}
\address{
Peking University\\
Beijing, 100871\\
P. R. China}
\email{binzhouxia@pku.edu.cn}


\maketitle

\begin{abstract}
This paper classifies quasiprimitive permutation groups with a transitive subgroup which is isomorphic to $\A_n$ for some $n\geqslant5$.
\end{abstract}


\section{Introduction}

Throughout this paper, all groups are assumed to be finite. For a group $G$, the expression $G=HK$ with $H,K$ proper subgroups of $G$ is called a \emph{factorization} of $G$, and $H,K$ are called its \emph{factors}.
As a generalization of simple groups, we call a group \emph{almost simple} if its socle is nonabelian simple. There has been a lot of work on factorizations of almost simple groups. In 1987, Hering, Liebeck and Saxl~\cite{HLS} classified factorizations of exceptional groups of Lie type. A factorization $G=HK$ is called a \emph{maximal factorization} of $G$ if both $H,K$ are maximal subgroups of $G$. Later in 1990, Liebeck, Praeger and Saxl published the landmark work~\cite{LPS1990} classifying maximal factorizations of almost simple groups. When the socle is an alternating group, in fact all the factorizations of such a group were determined in~\cite[THEOREM~D]{LPS1990}. Based on the maximal factorizations in~\cite[THEOREM~C]{LPS1990}, Giudici~\cite{Giudici} in 2006 determined the factorizations of sporadic groups. A factorization is said to be \emph{exact} if the intersection of the two factors is trivial. Recently, all exact factorizations of almost simple groups for which one factor is a maximal subgroup have been determined~\cite{LPS2010} by Liebeck, Praeger and Saxl.

Our first result classifies factorizations of finite simple groups with an alternating group factor.

\begin{theorem}\label{thm1}
Suppose that $L=HK$ is a factorization of the finite simple group $L$ with $H\cong\A_n$, where $n\geqslant5$. Then one of the following holds.
\begin{itemize}
\item[(a)] $L=\A_{n+k}$ for $1\leqslant k\leqslant5$, and $K$ is $k$-transitive on $n+k$ points.
\item[(b)] $L=\A_m$ and $K=\A_{m-1}$, where $m$ is the index of a subgroup in $\A_n$.
\item[(c)] $(L,n,K)$ lies in \emph{Table~\ref{tab1}}.
\end{itemize}
\end{theorem}

\begin{table}[htbp]
\caption{}\label{tab1}
\centering
\begin{tabular}{|l|l|l|l|}
\hline
row & $L$ & $n$ & $K$\\
\hline
1 & $\A_6$ & $5$ & $\A_4$, $\Sy_4$ \\
2 & $\A_{10}$ & $6$ & $\A_8$, $\Sy_8$ \\
3 & $\A_{15}$ & $7$ & $\A_{13}$, $\Sy_{13}$ \\
\hline
4 & $\M_{12}$ & $5$ & $\M_{11}$ \\
\hline
5 & $\PSL_2(11)$ & $5$ & $11$, $11{:}5$ \\
6 & $\PSL_2(19)$ & $5$ & $19{:}9$ \\
7 & $\PSL_2(29)$ & $5$ & $29{:}7$, $29{:}14$ \\
8 & $\PSL_2(59)$ & $5$ & $59{:}29$ \\
\hline
9 & $\PSL_4(3)$ & $6$ & $3^3{:}\PSL_3(3)$ \\
\hline
10 & $\PSU_3(5)$ & $7$ & $5_+^{1+2}{:}8$ \\
\hline
11 & $\PSp_4(3)$ & $6$ & $3_+^{1+2}{:}\Q_8$, $3_+^{1+2}{:}2.\A_4$ \\
\hline
12 & $\Sp_6(2)$ & $6$ & $\PSU_3(3){:}2$ \\
13 & $\Sp_6(2)$ & $7$ & $\PSU_3(3){:}2$ \\
\hline
\multirow{2}*{14} & \multirow{2}*{$\Sp_6(2)$} & \multirow{2}*{$8$} & $3_+^{1+2}{:}8{:}2$, $3_+^{1+2}{:}2.\Sy_4$, $\PSL_2(8)$, \\
 &  &  & $\PSL_2(8){:}3$, $\PSU_3(3){:}2$, $\PSU_4(2){:}2$ \\
\hline
15 & $\Sp_8(2)$ & $6$ & $\SO_8^-(2)$ \\
16 & $\Sp_8(2)$ & $7$ & $\SO_8^-(2)$ \\
17 & $\Sp_8(2)$ & $8$ & $\SO_8^-(2)$ \\
18 & $\Sp_8(2)$ & $9$ & $\SO_8^-(2)$ \\
19 & $\Sp_8(2)$ & $10$ & $\SO_8^-(2)$ \\
\hline
20 & $\Omega_7(3)$ & $8$ & $3^{3+3}{:}\PSL_3(3)$ \\
\hline
\multirow{2}*{21} & \multirow{2}*{$\Omega_7(3)$} & \multirow{2}*{$9$} & $3^3{:}\PSL_3(3)$, $3^{3+3}{:}\PSL_3(3)$ \\
 &  &  & $\PSL_4(3)$, $\PSL_4(3){:}2$, $\G_2(3)$ \\
\hline
22 & $\Omega_8^+(2)$ & $6$ & $\Sp_6(2)$ \\
23 & $\Omega_8^+(2)$ & $7$ & $\Sp_6(2)$ \\
\hline
\multirow{2}*{24} & \multirow{2}*{$\Omega_8^+(2)$} & \multirow{2}*{$8$} & $\Sp_6(2)$, $\PSU_4(2)$, $\PSU_4(2){:}2$, \\
 &  &  & $3\times\PSU_4(2)$, $(3\times\PSU_4(2)){:}2$, $\A_9$ \\
\hline
\multirow{3}*{25} & \multirow{3}*{$\Omega_8^+(2)$} & \multirow{3}*{$9$} & $2^4{:}15.4$, $2^6{:}15$, $2^6{:}15.2$, $2^6{:}15.4$, $\PSU_4(2)$, \\
 &  &  & $\PSU_4(2){:}2$, $3\times\PSU_4(2)$, $(3\times\PSU_4(2)){:}2$, \\
 &  &  & $\Sp_6(2)$, $\A_8$, $\Sy_8$ or $2^4{:}\A_5\leqslant K\leqslant2^6{:}\A_8$ \\
\hline
26 & $\POm_8^+(3)$ & $8$ & $3^6{:}\PSL_4(3)$ \\
27 & $\POm_8^+(3)$ & $9$ & $\Omega_7(3)$, $3^6{:}\PSL_4(3)$ \\
\hline
28 & $\Omega_{10}^-(2)$ & $12$ & $2^8{:}\Omega_8^-(2)$ \\
\hline
\end{tabular}
\end{table}

\begin{remark}
For each pair $(L,K)$ as described in part~(a) or~(b) of Theorem~\ref{thm1}, there holds $L=HK$ with $H\cong\A_n$. Moreover, it is verified by \magma~\cite{magma} that each triple $(L,n,K)$ in Table~\ref{tab1} gives rise to a factorization $L=HK$ with $H\cong\A_n$. The factorization $L=HK$ in row~25 of Table~\ref{tab1} with $(L,H,K)=(\Omega_8^+(2),\A_9,\Sy_8)$ is missed in the main theorem of~\cite{Darafsheh}.
\end{remark}

%
%

By Lemma~\ref{lem1} in Section~\ref{sec1}, one easily derives the following corollary from Theorem~\ref{thm1}.

\begin{corollary}\label{cor1}
Suppose that $H$ and $K$ are subgroups of a finite group $G$ with $G=HK$ and $H\cong\A_n$, where $n\geqslant5$. Then there exists a composition factor $L$ of $G$ such that one of the following holds.
\begin{itemize}
\item[(a)] $L=\A_{n+k}$ for $1\leqslant k\leqslant5$.
\item[(b)] $L=\A_m$, where $m$ is the index of a subgroup in $\A_n$.
\item[(c)] $(L,n)$ lies in \emph{Table~\ref{tab1}}.
\end{itemize}
\end{corollary}

In~\cite{LPS2010} all primitive groups having a regular subgroup which is almost simple are classified. This has a surprising consequence regarding Cayley graphs of simple groups: if a Cayley graph of some simple group is vertex primitive, then either the Cayley graph is a complete graph or a graph associated with the action of the alternating groups on $2$-subsets, or the connecting set of the Cayley graph consists of several full conjugacy classes. Inspired by the result of~\cite{LPS2010}, one might expect a classification of for primitive groups having a transitive subgroup which is almost simple. Although a complete classification seems to be fairly difficult, it is possible to achieve such a goal for certain families of almost simple groups, and the present paper will tackle this for alternating groups. In fact, our main result, Theorem~\ref{thm2} below, classifies quasiprimitive groups containing a transitive alternating group. Here a permutation group $G$ is called \emph{quasiprimitive} if every nontrivial normal subgroup of $G$ is transitive. The notion of quasiprimitive groups is a generalization of primitive groups, which arises naturally for instance in the analysis of $2$-arc transitive graphs, see~\cite{Praeger1993}. As a result, Theorem~\ref{thm2} will be crucial in classifying $2$-arc transitive graphs admitting a vertex transitive alternating group, especially, $2$-arc-transitive Cayley graph of alternating groups.

For a group $T$, denote by $D(2,T)$ the permutation group on $T$ generated by the holomorph of $T$ and the involution $x\mapsto x^{-1}$, $x\in T$.

\begin{theorem}\label{thm2}
Let $G$ be a quasiprimitive permutation group on $\Omega$ and $\alpha\in\Omega$. If $G$ contains a transitive subgroup $H$ such that $H\cong\A_n$ for some $n\geqslant5$, then one of the following holds.
\begin{itemize}
\item[(a)] $G$ is almost simple with socle $L$, say, and $H\leqslant L$. Furthermore, either $L=\A_n$, or the factorization $L=HL_\alpha$ is described in \emph{Theorem~\ref{thm1}}.
\item[(b)] $G\leqslant D(2,\A_n)$ is primitive with socle $\A_n\times\A_n$, and $H$ is regular.
\item[(c)] $n=6$, $G$ is primitive with socle $\A_6\times\A_6$, and $G\leqslant\Sy_6\wr\Sy_2$ by the product action on $6^2$ points.
\end{itemize}
\end{theorem}

The proof of Theorem~\ref{thm1} is given in Section~\ref{sec2}. Then in Section~\ref{sec3}, we prove Theorem~\ref{thm2} based on the result of Theorem~\ref{thm1}.

\section{Preliminaries}\label{sec1}

The notation in this paper is quite standard: it follows \cite{atlas} except for some classical simple groups. We will need the detailed information, such as the maximal subgroups, of some finite simple groups, for which the reader is also referred to \cite{atlas}. Hereafter, all the groups are assumed to be finite.

We first give several equivalent conditions for a group factorization.

\begin{lemma}\label{lem6}
Let $H,K$ be subgroups of $G$. Then the following are equivalent.
\begin{itemize}
\item[(a)] $G=HK$.
\item[(b)] $G=H^xK^y$ for any $x,y\in G$.
\item[(c)] $|H\cap K||G|=|H||K|$.
\item[(d)] $|G|\leqslant|H||K|/|H\cap K|$.
\item[(e)] $H$ acts transitively on $[G{:}K]$ by right multiplication.
\item[(f)] $K$ acts transitively on $[G{:}H]$ by right multiplication.
\end{itemize}
\end{lemma}

The above lemma is easy to prove while playing a fundamental role in the study of group factorizations: due to part~(b) we will consider conjugacy classes of subgroups when studying factorizations of a group; given a group $G$ and its subgroups $H,K$, in order to inspect whether $G=HK$ holds we only need to compute the orders of $G$, $H$, $K$ and $H\cap K$ by part~(c) or~(d), which enables us to search factorizations of a group rather quickly in \magma~\cite{magma}, for example.

The next lemma, applied together with Jordan-H\"{o}lder Theorem, shows that it is crucial to study factorizations of simple groups to study factorizations of a general group with a simple factor.

\begin{lemma}\label{lem1}
Suppose that $H,K$ are subgroup of $G$ such that $G=HK$ and $H$ is simple. Then for any normal subgroup $N$ of $G$, either $N=H(K\cap N)$ or $G/N=\overline{H}\,\overline{K}$ for some $\overline{H}\cong H$ and $\overline{K}\cong K/(K\cap N)$.
\end{lemma}

\proof
As $N$ is normal in $G$ and $H$ is simple, we have $H\cap N=1$ or $H$. First assume that $H\cap N=1$. Let $\overline{H}=HN/N$ and $\overline{K}=KN/N$. Then $\overline{H}\cong H$ and $\overline{K}\cong K/(K\cap N)$. Next assume that $H\cap N=H$. Then combination of the equality $G=HK$ and condition
$H\leqslant N$ deduces that $N=H(K\cap N)$.
\qed

%
%

%

\section{Proof of Theorem~\ref{thm1}}\label{sec2}

\begin{lemma}\label{lem2}
If $L$ is a non-classical simple group, then \emph{Theorem~\ref{thm1}} holds.
\end{lemma}

\proof
Let $L$ be a non-classical simple group satisfying the assumption of Theorem~\ref{thm1}. Consulting the classification of factorizations of exceptional groups of Lie type \cite{HLS}, one sees that $L$ is not an exceptional group of Lie type. If $L$ is a sporadic simple group, then $(L,n,K)$ lies in row~4 of Table~\ref{tab1} by \cite{Giudici}. Next assume that $L=\A_m$ acting naturally on a set $\Omega$ of $m$ points. Then $m\geqslant6$, and according to Theorem~D and Remark~2 after it in \cite{LPS1990} one of the following cases appears.
\begin{itemize}
\item[(i)] $\A_{m-k}\leqslant H\leqslant\Sy_{m-k}\times\Sy_k$ and $K$ is $k$-homogeneous on $\Omega$ for some $1\leqslant k\leqslant5$.
\item[(ii)] $H$ is $k$-homogeneous on $\Omega$ and $\A_{m-k}\leqslant K\leqslant\Sy_{m-k}\times\Sy_k$ for some $1\leqslant k\leqslant5$.
\item[(iii)] $n=6$, $H=\PSL_2(5)$ and $K\leqslant\Sy_3\wr\Sy_2$ with $H,K$ both transitive on $\Omega$.
\end{itemize}
Note that case~(iii) is described in part~(a) of Theorem~\ref{thm1} with $n=5$ and $k=1$.

Assume that case~(i) appears. Then $H=\A_{m-k}$ and $m=n+k$. Since $H$ is the stabilizer of a $k$-tuple over $\Omega$, it follows that $K$ is $k$-transitive on $\Omega$, as part~(a) of Theorem~\ref{thm1}.

Next assume that case~(ii) appears. If $k=1$ then $K=\A_{m-1}$ and $H$ is transitive on $\Omega$, which leads to part~(b) of Theorem~\ref{thm1}. Thus assume $k\geqslant2$. Then \cite{Kantor} shows that $H$ is $k$-transitive on $\Omega$, and inspecting the classification of $2$-transitive permutation groups we conclude that $k=2$ and $(m,n)=(6,5)$, $(10,6)$ or $(15,7)$. Consequently,
$$
\A_{m-2}\leqslant K\leqslant(\Sy_{m-2}\times\Sy_2)\cap\A_m=\Sy_{m-2},
$$
and so $(L,n,K)$ lies in one of rows~1--3 of Table~\ref{tab1}.
\qed

\begin{lemma}\label{lem5}
If $H\cong\A_5$, then \emph{Theorem~\ref{thm1}} holds.
\end{lemma}

\proof
Let $M$ be a maximal subgroup of $L$ containing $K$. It derives from the factorization $L=HK$ that $L=HM$, and so $L$ acts primitively on the set $[L{:}M]$ of right cosets of $M$ in $L$. By Lemma~\ref{lem6}, $|L|/|M|$ divides $|H|=60$. Hence $L$ is a primitive permutation group of degree dividing $60$, and thus lies in \cite[Appendix~B]{DM}. Due to Lemma~\ref{lem2} we may assume that $L$ is a classical group, whence $L=\PSL_2(p)$ with $p\in\{11,19,29,59\}$ and $M=\Z_p{:}\Z_{(p-1)/2}$.

If $p=11$, then $|K|$ is divisible by $|L|/|H|=11$ and so $K=\Z_{11}$ or $\Z_{11}{:}\Z_5$ as in row~5 of Table~\ref{tab1}. Similarly, the cases $p=29$ and $p=59$ lead to rows~7 and~8 of Table~\ref{tab1}. Now assume that $p=19$. Then $|K|$ is divisible by $|L|/|H|=19\cdot3$. Since $\PSL_2(19)$ has only one conjugacy class of subgroups of order $3$, we exclude by Lemma~\ref{lem6} the possibility for $K=\Z_{19}{:}\Z_3$. Consequently, $K=\Z_{19}{:}\Z_9$ as in row~6 of Table~\ref{tab1}.
\qed

In order to prove Theorem~\ref{thm1}, we analyze its minimal counterexample in the rest of this section. To be precise, make the hypothesis as follows.

\begin{hypothesis}\label{hyp1}
Suppose that $(L,H,K)$ is a counterexample to Theorem~\ref{thm1} with minimal $|L|$. Let $A$ and $B$ be maximal subgroups of $L$ containing $H$ and $K$, respectively.
\end{hypothesis}

We remark that, under Hypothesis~\ref{hyp1}, $L$ is a classical simple group not isomorphic to $\A_5$, $\A_6$ or $\A_8$ due to Lemma~\ref{lem2}, and interchanging $A$ and $B$ if necessary, the triple $(L,A,B)$ lies in Tables~1--4 of \cite{LPS1990}.

\begin{lemma}\label{lem3}
Under \emph{Hypothesis~\ref{hyp1}}, $A$ has exactly one nonabelian composition factor.
\end{lemma}

\proof
Apparently, $A$ has at least one nonabelian composition factor. Suppose for a contradiction that $A$ has at least two nonabelian composition factors. Then inspecting Tables~1--4 of \cite{LPS1990}, we obtain Table~\ref{tab2} for candidates of $(L,A,B)$, where $a\leqslant2$ and $\N_1$ is the stabilizer of a non-isotropic $1$-dimensional subspace.

\begin{table}[htbp]
\caption{}\label{tab2}
\centering
\begin{tabular}{|l|l|l|l|}
\hline
row & $L$ & $A$ & $B$\\
\hline
1 & $\Sp_{4\ell}(2^f)$, $f\ell\geqslant2$ & $\Sp_{2\ell}(2^f)\wr\Sy_2$ & $\SO_{4\ell}^-(2^f)$\\
2 & $\Sp_4(2^f)$, $f\geqslant2$ & $\SO_4^+(2^f)$ & $\Sp_2(4^f).2$\\
3 & $\Sp_4(2^f)$, $f\geqslant3$ odd & $\SO_4^+(2^f)$ & $\Sz(2^f)$\\
4 & $\Sp_6(2^f)$, $f\geqslant2$ & $\Sp_2(2^f)\times\Sp_4(2^f)$ & $\G_2(2^f)$\\
\hline
5 & $\POm_{4\ell}^+(q)$, $\ell\geqslant3$, $q\geqslant4$ & $(\PSp_2(q)\times\PSp_{2\ell}(q)).a$ & $\N_1$\\
6 & $\POm_8^+(q)$, $q\geqslant5$ odd & $(\PSp_2(q)\times\PSp_4(q)).2$ & $\Omega_7(q)$\\
7 & $\Omega_8^+(2)$ & $(\SL_2(4)\times\SL_2(4)).2^2$ & $\Sp_6(2)$\\
\hline
\end{tabular}
\end{table}

From the inclusion $H\leqslant A$ and factorization $L=AB$ we deduce that $A=H(K\cap B)$. By Lemma~\ref{lem1}, there exists a nonabelian composition factor $L_1$ of $A$ such that $L_1=H_1K_1$ for some subgroups $H_1$ and $K_1$ of $L_1$ with $H_1\cong H$. Since $|L_1|\leqslant|A|<|L|$, Hypothesis~\ref{hyp1} implies that either $H_1=L_1$ or the triple $(L_1,H_1,K_1)$ satisfies the conclusion of Theorem~\ref{thm1}. This yields the following statements.
\begin{itemize}
\item[(a)] If $(L,A,B)$ lies in row~1 of Table~\ref{tab2}, then $(\ell,f)=(1,2)$ and $H\cong H_1=L_1\cong\A_5$.
\item[(b)] If $(L,A,B)$ lies in row~2 of Table~\ref{tab2}, then $f=2$ and $H\cong H_1=L_1\cong\A_5$.
\item[(c)] If $(L,A,B)$ lies in row~5 of Table~\ref{tab2}, then $q=4$ or $5$ and $H\cong H_1=L_1\cong\A_5$.
\item[(d)] If $(L,A,B)$ lies in row~6 of Table~\ref{tab2}, then $q=5$ and $H\cong H_1=L_1\cong\A_5$.
\item[(e)] If $(L,A,B)$ lies in row~7 of Table~\ref{tab2}, then $H\cong H_1=L_1\cong\A_5$.
\item[(f)] $(L,A,B)$ does not lie in rows~3 and~4 of Table~\ref{tab2}.
\end{itemize}
Now as we always have $H\cong\A_5$, the triple satisfies the conclusion of Theorem~\ref{thm1} by Lemma~\ref{lem5}, violating Hypothesis~\ref{hyp1}.
\qed

\begin{lemma}\label{lem4}
Under \emph{Hypothesis~\ref{hyp1}}, the unique nonabelian composition factor of $A$ is not an alternating group.
\end{lemma}

\proof
Suppose on the contrary that the nonabelian composition factor of $A$ is an alternating group, say $\A_m$, where $m\geqslant5$. Then in view of the isomorphisms
$$
\PSL_2(4)\cong\PSL_2(5)\cong\A_5,\quad\PSL_2(9)\cong\A_6,\quad\PSL_4(2)\cong\A_8\quad\text{and}\quad\Sp_4(2)\cong\Sy_6,
$$
an inspection of Tables~1--4 of \cite{LPS1990} shows that $(L,A,B)$ lies in Table~\ref{tab3}.

\begin{table}[htbp]
\caption{}\label{tab3}
\centering
\begin{tabular}{|l|l|l|l|}
\hline
row & $L$ & $A$ & $B$\\
\hline
1 & $\PSL_4(3)$ & $(4\times\A_6){:}2$ & $3^3{:}\PSL_3(3)$ \\
2 & $\PSL_3(5)$ & $5^2{:}\GL_2(5)$ & $31{:}3$ \\
3 & $\PSL_3(9)$ & $3^4{:}\GL_2(9)$ & $91{:}3$ \\
4 & $\PSL_5(2)$ & $2^4{:}\A_8$ & $31{:}5$ \\
5 & $\PSL_2(11)$ & $\A_5$ & $11{:}5$ \\
6 & $\PSL_2(19)$ & $\A_5$ & $19{:}9$ \\
7 & $\PSL_2(29)$ & $\A_5$ & $29{:}14$ \\
8 & $\PSL_2(59)$ & $\A_5$ & $59{:}29$ \\
\hline
9 & $\PSp_4(3)$ & $\Sy_6$ & $3_+^{1+2}{:}2.\A_4$ \\
10 & $\PSp_4(3)$ & $2^4{:}\A_5$ & $3_+^{1+2}{:}2.\A_4$, $3^3{:}\Sy_4$ \\
11 & $\Sp_4(4)$ & $2^6{:}(3\times\A_5)$ & $\PSL_2(16){:}2$ \\
12 & $\PSp_4(5)$ & $5_+^{1+2}{:}4.\A_5$ & $\PSL_2(25){:}2$ \\
13 & $\PSp_4(9)$ & $3^{2+4}{:}8.\A_6$ & $\PSL_2(81){:}2$ \\
14 & $\Sp_6(2)$ & $2^5{:}\Sy_6$ & $\SL_2(8){:}3$ \\
15 & $\Sp_6(2)$ & $\Sy_8$ & $\SL_2(8){:}3$, $\PSU_4(2){:}2$ \\
16 & $\Sp_6(2)$ & $\Sy_8$, $2^5{:}\Sy_6$, $\Sy_3\times\Sy_6$ & $\PSU_3(3){:}2$ \\
17 & $\Sp_8(2)$ & $\Sy_{10}$ & $\SO_8^-(2)$ \\
\hline
18 & $\PSU_4(3)$ & $3^4{:}\A_6$ & $\PSU_3(3)$ \\
19 & $\PSU_3(5)$ & $\A_7$ & $5_+^{1+2}{:}8$ \\
\hline
20 & $\Omega_7(3)$ & $\Sy_9$ & $3^{3+3}{:}\PSL_3(3)$, $\PSL_4(3){:}2$, $\G_2(3)$ \\
21 & $\Omega_7(3)$ & $2^6{:}\A_7$ & $\G_2(3)$ \\
\hline
22 & $\Omega_{10}^-(2)$ & $\A_{12}$ & $2^8{:}\Omega_8^-(2)$ \\
\hline
23 & $\Omega_8^+(2)$ & $2^6{:}\A_8$ & $\Sp_6(2)$, $(3\times\PSU_4(2)){:}2$, $\A_9$ \\
24 & $\Omega_8^+(2)$ & $\A_9$ & $\Sp_6(2)$, $2^6{:}\A_8$, $(3\times\PSU_4(2)){:}2$ \\
\hline
\end{tabular}
\end{table}

Notice that $H$ is isomorphic to a subgroup of $\A_m$ and $H\ncong\A_5$ by Lemma~\ref{lem5}, whence $6\leqslant n\leqslant m$. Therefore, we only need to consider rows~1, 3, 4, 9 and~13--24 of Table~\ref{tab3}. Since $|L|$ divides $|H||K|$ by Lemma~\ref{lem6}, one deduces that $|L|/|B|$ divides $m!/2$, which excludes rows~3, 4, 13, 14, 18 and~21. We next analyze the remaining rows, i.e., rows~1, 9, 15--17, 19, 20 and~22--24.

Assume that row~1 of Table~\ref{tab3} appears. Then $m=6$, and hence $n=6$. It follows that $|K|$ is divisible by $|L|/|H|=|\PSL_4(3)|/|\A_6|=2^4\cdot3^4\cdot13$. Write $B=R{:}Q$ with $R=\Z_3^3$ and $Q=\PSL_3(3)$. Then $K\cap R\geqslant\Z_3$ as $3^4$ does not divide $|Q|$, and $KR/R$ is a subgroup of $\PSL_3(3)$ with order divisible by $2^4\cdot13$, which forces $KR/R=\PSL_3(3)$. Since here $\PSL_3(3)$ acts irreducibly on $\Z_3^3$, we conclude that $K=B=3^3{:}\PSL_3(3)$, as in row~9 of Table~\ref{tab1}.

Assume that row~9 of Table~\ref{tab3} appears. Then $n=m=6$, and searching in \magma~\cite{magma} for the factor corresponding $K$ leads to row~11 of Table~\ref{tab1}.

If $(L,A,B)$ lies in row~15 or~16 of Table~\ref{tab3}, then $6\leqslant n\leqslant8$. Setting $n=6$, $7$, $8$, respectively, and searching in \magma~\cite{magma} for the corresponding subgroup $K<\Sp_6(2)$ satisfying $L=HK$ shows that one of rows~12--14 of Table~\ref{tab1} appears.

Assume that row~17 of Table~\ref{tab3} appears. In this case, $6\leqslant n\leqslant10$. However, one can quickly verifies by \magma~\cite{magma} that there is no factorization $L=HB$ with $(L,H,B)=(\Sp_8(2),\A_n,\SO_8^-(2))$ for $6\leqslant n\leqslant9$. As a consequence, $n=10$. Then searching in \magma~\cite{magma} for the corresponding factor $K$ leads to row~19 of Table~\ref{tab1}.

Assume that row~19 of Table~\ref{tab3} appears. Then $n!/2=|H|$ is divisible by $|L|/|B|=126$. This in conjunction with the inequality $n\leqslant m=7$ yields $n=7$. Now searching in \magma~\cite{magma} for the corresponding factor $K$ one sees that $(L,H,K)$ lies in row~10 of Table~\ref{tab1}.

Let $(L,A,B)$ be as described in row~20 of Table~\ref{tab3}. Then $n\leqslant m=9$, and $n!/2=|H|$ is divisible by $|L|/|B|$. If $B=\PSL_4(3){:}2$ or $\G_2(3)$, then this forces $n=9$, and searching in \magma~\cite{magma} for the corresponding factor $K$ leads to row~21 of Table~\ref{tab1}. If $B=3^{3+3}{:}\PSL_3(3)$, then $n=8$ or $9$, and searching in \magma~\cite{magma} for the corresponding factor $K$ leads to row~20 or~21 of Table~\ref{tab1}.

Now consider row~22 of Table~\ref{tab3}. In this case, $6\leqslant n\leqslant12$, and $n!/2=|H|$ is divisible by $|L|/|B|=3^2\cdot5\cdot11$. However, one can quickly verifies by \magma~\cite{magma} that there is no factorization $L=HB$ with $(L,H,B)=(\Omega_{10}^-(2),\A_{11},2^8{:}\Omega_8^-(2))$. As a consequence, $n=12$. Then searching in \magma~\cite{magma} for the corresponding factor $K$ leads to row~28 of Table~\ref{tab1}.

Finally, assume that $(L,A,B)$ lies in row~23 or~24 of Table~\ref{tab3}. Then we have $6\leqslant n\leqslant9$. Setting $n=6$, $7$, $8$, $9$, respectively, and searching in \magma~\cite{magma} for the corresponding subgroup $K<\Omega_8^+(2)$ satisfying $L=HK$ shows that one of rows~22--25 of Table~\ref{tab1} appears.
\qed

Now we are able to complete the proof for Theorem~\ref{thm1}.

\noindent\textbf{Proof of Theorem~\ref{thm1}:} Suppose Hypothesis~\ref{hyp1}. Then $A$ has a unique nonabelian composition factor $L_1$ by Lemma~\ref{lem3}, and $L_1$ is not an alternating group according to Lemma~\ref{lem4}. From the inclusion $H\leqslant A$ and factorization $L=AB$ we deduce that $A=H(K\cap B)$. Hence by Lemma~\ref{lem1}, $L_1=H_1K_1$ for some subgroups $H_1$ and $K_1$ of $L_1$ with $H_1\cong H$. Since $|L_1|\leqslant|A|<|L|$, Hypothesis~\ref{hyp1} implies that either $H_1=L_1$ or the triple $(L_1,H_1,K_1)$ satisfies the conclusion of Theorem~\ref{thm1}. Moreover, $L_1$ is not an alternating group, and by Lemma~\ref{lem5}, $H_1\ncong\A_5$. Then one sees from Table~\ref{tab1} that $L_1$ is one of the groups:
$$
\PSL_4(3),\ \PSU_3(5),\ \PSp_4(3),\ \Sp_6(2),\ \Sp_8(2),\ \Omega_7(3),\ \Omega_8^+(2),\ \POm_8^+(3),\ \Omega_{10}^-(2).
$$
Inspecting Tables~1--4 of \cite{LPS1990}, we thereby deduce that $(L,A,B)$ lies in Table~\ref{tab4}, where $a\leqslant2$.

\begin{table}[htbp]
\caption{}\label{tab4}
\centering
\begin{tabular}{|l|l|l|l|}
\hline
row & $L$ & $A$ & $B$\\
\hline
1 & $\PSL_5(3)$ & $3^4{:}\GL_4(3)$ & $121{:}5$ \\
2 & $\PSL_4(3)$ & $\PSp_4(3){:}2$ & $3^3{:}\PSL_3(3)$ \\
3 & $\PSL_6(2)$ & $\Sp_6(2)$ & $2^5{:}\PSL_5(2)$ \\
4 & $\PSL_8(2)$ & $\Sp_8(2)$ & $2^7{:}\PSL_7(2)$ \\
\hline
5 & $\PSp_6(3)$ & $3_+^{1+4}{:}2.\PSp_4(3)$ & $\PSL_2(27){:}3$, $\PSL_2(13)$ \\
6 & $\Sp_8(2)$ & $2^7{:}\Sp_6(2)$, $\SO_8^+(2)$, $\Sy_3\times\Sp_6(2)$ & $\Sp_4(4){:}2$ \\
7 & $\Sp_8(2)$ & $\SO_8^+(2)$ & $\SO_8^-(2)$, $\PSL_2(17)$ \\
8 & $\Sp_{10}(2)$ & $2^9{:}\Sp_8(2)$ & $\SL_2(32){:}5$ \\
9 & $\Sp_{10}(2)$ & $\SO_{10}^-(2)$ & $\SL_2(32){:}5$, $2^{15}{:}\SL_5(2)$, $\SO_{10}^+(2)$ \\
\hline
10 & $\PSU_4(5)$ & $\SU_3(5){:}3$ & $5^4{:}\PSL_2(25).4$, $\PSp_4(5){:}2$ \\
11 & $\PSU_4(3)$ & $\PSp_4(3)$ & $\PSU_3(3)$, $\PSL_3(4)$ \\
12 & $\PSU_6(2)$ & $\Sp_6(2)$ & $\PSU_5(2)$ \\
13 & $\PSU_8(2)$ & $\Sp_8(2)$ & $\GU_7(2)$ \\
\hline
14 & $\Omega_9(3)$ & $3^{6+4}{:}\SL_4(3)$ & $\Omega_8^-(3).2$ \\
\hline
\multirow{2}*{15} & \multirow{2}*{$\Omega_7(3)$} & $3^5{:}\PSp_4(3){:}2$, $\PSL_4(3){:}2$, & \multirow{2}*{$\G_2(3)$} \\
 &  & $(2^2\times\PSp_4(3)){:}2$, $\Sp_6(2)$ & \\
\hline
16 & $\Omega_7(3)$ & $\PSL_4(3){:}2$ & $\Sy_9$ \\
\hline
\multirow{2}*{17} & \multirow{2}*{$\POm_8^+(3)$} & \multirow{2}*{$\Omega_7(3)$} & $\Omega_7(3)$, $2.\PSU_4(3).2^2$, $\Omega_8^+(2)$, \\
 &  &  & $3^6{:}\PSL_4(3)$, $(\A_4\times\PSp_4(3)){:}2$ \\
\hline
18 & $\POm_8^+(3)$ & $3^6{:}\PSL_4(3)$ & $\Omega_7(3)$, $2.\PSU_4(3).2^2$, $\Omega_8^+(2)$ \\
19 & $\POm_8^+(3)$ & $(\A_4\times\PSp_4(3)){:}2$ & $\Omega_7(3)$ \\
20 & $\POm_8^+(3)$ & $\Omega_8^+(2)$ & $\Omega_7(3)$, $3^6{:}\PSL_4(3)$ \\
\hline
21 & $\Omega_8^+(2)$ & $\Sp_6(2)$ & $\A_9$, $(\A_5\times\A_5){:}2^2$ \\
22 & $\Omega_8^+(2)$ & $(3\times\PSp_4(3)){:}2$ & $\A_9$ \\
\hline
23 & $\Omega_{16}^+(2)$ & $\Sp_8(2).a$ & $\Sp_{14}(2)$ \\
\hline
\end{tabular}
\end{table}

To finish the proof, we analyze the rows of Table~\ref{tab4} in order.

\underline{Row~1.} In this case, $L_1=\PSL_4(3)$. Then as $(L_1,H_1,K_1)$ satisfies the conclusion of Theorem~\ref{thm1}, $H\cong H_1\cong\A_6$. However, $|L|$ is divisible by $13$ while $|B|$ is not. We thereby deduce from the factorization $L=HB$ that $|H|$ is divisible by $13$, a contradiction.

\underline{Row~2.} Here $L_1=\PSp_4(3)$, and as $(L_1,H_1,K_1)$ satisfies the conclusion of Theorem~\ref{thm1}, $H\cong H_1\cong\A_6$. It follows that $|K|$ is divisible by $|L|/|H|=|\PSL_4(3)|/|\A_6|=2^4\cdot3^4\cdot13$. Write $B=R{:}Q$ with $R=\Z_3^3$ and $Q=\PSL_3(3)$. Then $K\cap R\geqslant\Z_3$ as $3^4$ does not divide $|Q|$, and $KR/R$ is a subgroup of $\PSL_3(3)$ with order divisible by $2^4\cdot13$, which forces $KR/R=\PSL_3(3)$. Since here $\PSL_3(3)$ acts irreducibly on $\Z_3^3$, we conclude that $K=B=3^3{:}\PSL_3(3)$, as in row~11 of Table~\ref{tab1}.

\underline{Row~3.} As $(L_1,H_1,K_1)$ satisfies the conclusion of Theorem~\ref{thm1}, where $L_1=\Sp_6(2)$, we have $6\leqslant n\leqslant8$. However, one can quickly verifies by \magma~\cite{magma} that there is no factorization $L=HB$ with $(L,H,B)=(\PSL_6(2),\A_n,2^5{:}\PSL_5(2))$ for $6\leqslant n\leqslant8$.

\underline{Row~4.} Here $L_1=\Sp_8(2)$, and since $(L_1,H_1,K_1)$ satisfies the conclusion of Theorem~\ref{thm1} we have $6\leqslant n\leqslant10$. However, $|L|$ is divisible by $17$ while $|B|$ is not. We then deduce from the factorization $L=HB$ that $|H|$ is divisible by $17$, a contradiction.

\underline{Row~5.} In this case, $L_1=\PSp_4(3)$. Then as $(L_1,H_1,K_1)$ satisfies the conclusion of Theorem~\ref{thm1}, $H\cong H_1\cong\A_6$. Since $|L|/|B|$ is divisible by $3^5$, we deduce from the factorization $L=HB$ that $|H|$ is divisible by $3^5$, a contradiction.

\underline{Rows~6 and~7.} For these two rows, $L=\Sp_8(2)$ and $L_1=\Sp_6(2)$ or $\Omega_8^+(2)$. If $L_1=\Sp_6(2)$, then $B=\Sp_4(4){:}2$ and $6\leqslant n\leqslant8$ since $(L_1,H_1,K_1)$ satisfies the conclusion of Theorem~\ref{thm1}. As $|L|/|B|=2^7\cdot3^3\cdot7$ does not divide $8!/2$, we thereby exclude the possibility for $L_1=\Sp_6(2)$. Consequently, $L_1=\Omega_8^+(2)$, and so $6\leqslant n\leqslant9$ since $(L_1,H_1,K_1)$ satisfies the conclusion of Theorem~\ref{thm1}. Now $|L|/|B|$ divides $|H|$ and thus divides $9!/2$, which yields that $B=\SO_8^-(2)$. Setting $n=6$, $7$, $8$, $9$, respectively, and searching in \magma~\cite{magma} for the corresponding subgroup $K<\Sp_8(2)$ satisfying $L=HK$ shows that one of rows~15--18 of Table~\ref{tab1} appears.

\underline{Rows~8 and~9.} For these two rows, $L=\Sp_{10}(2)$ and $L_1=\Sp_8(2)$ or $\Omega_{10}^-(2)$. As $(L_1,H_1,K_1)$ satisfies the conclusion of Theorem~\ref{thm1}, we have $n\leqslant12$. If $B=\SL_2(32){:}5$ or $2^{15}{:}\SL_5(2)$, then $17$ divides $|L|$ but not $|B|$, which yields a contradiction that $17$ divides $|H|=n!/2$. Consequently, $B=\SO_{10}^+(2)$. It follows that $A=\SO_{10}^-(2)$, and $H\cong H_1\cong\A_{12}$ since $(L_1,H_1,K_1)$ satisfies the conclusion of Theorem~\ref{thm1}. However, searching in \magma~\cite{magma} shows that there is no factorization $L=HB$ with $(L,H,B)=(\Sp_{10}(2),\A_{12},\SO_{10}^+(2))$.

\underline{Row~10.} In this case, $L_1=\PSU_3(5)$. Then as $(L_1,H_1,K_1)$ satisfies the conclusion of Theorem~\ref{thm1}, $H\cong H_1\cong\A_7$. However, $|L|/|B|=2^2\cdot3^3\cdot7$ or $3^2\cdot5^2\cdot7$, not divisible by $7!/2$. Hence this case is not possible either.

\underline{Row~11.} As $(L_1,H_1,K_1)$ satisfies the conclusion of Theorem~\ref{thm1}, where $L_1=\PSp_4(3)$, we have $H\cong H_1\cong\A_6$. However, $3^3$ divides $|L|/|B|$ and thus divides $|H|$, which is a contradiction.

\underline{Row~12.} As $(L_1,H_1,K_1)$ satisfies the conclusion of Theorem~\ref{thm1}, where $L_1=\Sp_6(2)$, we have $6\leqslant n\leqslant8$. Since $2^5$ divides $|L|/|B|$ and thus divides $|H|=n!/2$, it follows that $n=8$. However, searching in \magma~\cite{magma} shows that there is no factorization $L=HB$ with $(L,H,B)=(\PSU_6(2),\A_8,\PSU_5(2))$.

\underline{Row~13.} As $(L_1,H_1,K_1)$ satisfies the conclusion of Theorem~\ref{thm1}, where $L_1=\Sp_8(2)$, we have $6\leqslant n\leqslant10$. However, $|L|$ is divisible by $17$ while $|B|$ is not. We thereby deduce from the factorization $L=HB$ that $|H|$ is divisible by $17$, a contradiction.

\underline{Row~14.} Here $L_1=\PSL_4(3)$, and as $(L_1,H_1,K_1)$ satisfies the conclusion of Theorem~\ref{thm1}, $H\cong H_1\cong\A_6$. However, $3^4$ divides $|L|/|B|$ and thus divides $|H|$, which is a contradiction.

\underline{Rows~15 and~16.} For these two rows, $L=\Omega_7(3)$ and $L_1=\PSp_4(3)$, $\PSL_4(3)$ or $\Sp_6(2)$. Then $6\leqslant n\leqslant8$ since $(L_1,H_1,K_1)$ satisfies the conclusion of Theorem~\ref{thm1}. However, $3^3$ divides $|L|/|B|$ and thus divides $|H|=n!/2$, which is a contradiction.

\underline{Rows~17--20.} For these four rows, $L=\POm_8^+(3)$ and $L_1=\Omega_7(3)$, $\PSL_4(3)$, $\PSp_4(3)$ or $\Omega_8^+(2)$. As $(L_1,H_1,K_1)$ satisfies the conclusion of Theorem~\ref{thm1}, $6\leqslant n\leqslant9$. Then checking the condition that $|L|/|B|$ divides $9!/2$ we obtain $B=\Omega_7(3)$ or $3^6{:}\PSL_4(3)$. First assume that $B=\Omega_7(3)$. We then immediately have $n=9$ because $|L|/|B|$ divides $|H|=n!/2$, and it follows that $|K|$ is divisible by $|L|/|H|=|\POm_8^+(3)|/|\A_9|=2^6\cdot3^8\cdot5\cdot13$. This forces $K=B=\Omega_7(3)$, as in row~27 of Table~\ref{tab1}. Next assume that $B=2^6{:}\PSL_4(3)$, and write $B=R{:}Q$ with $R=\Z_3^6$ and $Q=\PSL_4(3)$. Then $n=8$ or $9$ since $|L|/|B|$ divides $n!/2$, and $|K|$ is divisible by $|\POm_8^+(3)|/|\A_9|=2^6\cdot3^8\cdot5\cdot13$.
Noticing that $3^7$ does not divide $|Q|$, we have $K\cap R\geqslant\Z_3$. Moreover, $KR/R$ is a subgroup of $\PSL_4(3)$ with order divisible by $2^6\cdot5\cdot13$, which forces $KR/R=\PSL_4(3)$. Since here $\PSL_4(3)$ acts irreducibly on $\Z_3^6$, we conclude that $K=B=3^6{:}\PSL_4(3)$, as in row~26 or~27 of Table~\ref{tab1}.

\underline{Rows~21 and~22.} For these two rows, $L=\POm_8^+(2)$ and $L_1=\Sp_6(2)$ or $\PSp_4(3)$. As $(L_1,H_1,K_1)$ satisfies the conclusion of Theorem~\ref{thm1}, we have $6\leqslant n\leqslant8$. It follows that $|L|/|B|$ divides $8!/2$, which yields $B=\A_9$. This in turn gives $n=8$ because $n!/2$ is divisible by $|L|/|B|=|\Omega_8^+(2)|/|\A_9|=2^6\cdot3\cdot5$. Now $|K|$ is divisible by $|L|/|H|=|\Omega_8^+(2)|/|\A_8|=2^6\cdot3^3\cdot5$, so we deduce that $K=B=\A_9$, as in row~24 of Table~\ref{tab1}.

\underline{Row~23.} As $(L_1,H_1,K_1)$ satisfies the conclusion of Theorem~\ref{thm1}, where $L_1=\Sp_8(2)$, we have $6\leqslant n\leqslant10$. However, $17$ divides $|L|/|B|$ and thus divides $|H|$, which is a contradiction.
\qed

\section{Proof of Theorem~\ref{thm2}}\label{sec3}

Quasiprimitive permutation groups are divided into eight categories, namely, HA, HS, HC, AS, SD, CD, TW and PA, see~\cite{Praeger1996}. For the purpose of convenience, we include types HS and SD together as \emph{diagonal type}.

%
%
%

\begin{lemma}\label{lem7}
Let $p$ be prime, $k\geqslant1$ and $n\geqslant5$. Then $\AGL_k(p)$ acting naturally on $p^k$ points does not have any transitive subgroup isomorphic to $\A_n$.
\end{lemma}

\begin{proof}
Let $L$ be the socle of $\AGL_k(p)$, and suppose that $H$ is a transitive subgroup of $\AGL_k(p)$ isomorphic to $\A_n$. Then $L=\Z_p^k$, and $H\cap L=1$ or $H$ since $H$ is simple. As $H$ is unsolvable, $k\geqslant2$ and $H\nleqslant L$. It follows that $H\cap L=1$, and thus $H\cong HL/L\leqslant\AGL_k(p)/L\cong\GL_k(p)$. Then by~\cite[Proposition~5.3.7]{KL} we have $k\geqslant n-4$. Moreover, since $H$ is transitive on $p^k$ points, we have $n=p^k$ by~\cite{Guralnick}. Therefore, $k\leqslant p^k-4$. As a consequence, $k\geqslant2^k-4$, forcing $k=2$. This in turn gives $2\geqslant p^2-4$. Hence $p=2$ and $n=p^k=4$, contrary to the assumption that $n\geqslant5$.
\end{proof}

\noindent\textbf{Proof of Theorem~\ref{thm2}:} By Lemma~\ref{lem7}, $G$ has socle $T^k$ for some nonabelian simple group $T$ and integer $k\geqslant1$. First assume that $k=1$. Then $G/T\lesssim\Out(T)$ is solvable, and so $H\cap T\neq1$. Hence $H\cap T=H$, which means that $H$ is a transitive subgroup of $T$. It follows that $T=HT_\alpha$ and thus part~(a) of Theorem~\ref{thm2} holds.

Next assume that $k\geqslant2$. Take a $G$-invariant partition $\calB$ (may be trivial) of $\Omega$ such that $G$ acts primitively on $\calB$. Then $G$ acts faithfully on $\calB$ since $G$ is quasiprimitive on $\Omega$. Now $G^\calB$ is a primitive permutation group isomorphic to $G$, and $H^\calB$ is a quasiprimitive subgroup of $G^\calB$ since $H$ is simple and transitive on $\Omega$. Appealing to \cite{BP} we deduce $k=2$ and $H\cong T$. Therefore, $G$ is either of diagonal type or type PA. If $G$ is of diagonal type, then $H$ is regular and $G$ is primitive as $k=2$, and so part~(b) of Theorem~\ref{thm2} occurs according to~\cite{LPS2010}. Hence we assume that $G$ is of type PA in the following.

Under the above assumption, there exists a group $X$ isomorphic to $G$ such that $X\leqslant R\wr\Sy_2=N.\Sy_2$ with $R$ an almost simple quasiprimitive permutation group on a set $\Delta$, $N=R^2$ and $R$ has socle $T$. Also, the isomorphic image $Y$ of $H$ into $X$ is transitive on $\Delta^2$. Hence $Y\cong\A_n$, and $Y\cap N=1$ or $Y$. If $Y\cap N=1$, then $Y\leqslant\Sy_2$, a contradiction. Thus $Y\cap N=Y$, which means that $Y\leqslant N$. Write $N=R_1\times R_2$ in the natural way with $R_1\cong R_2\cong R$. Then $Y\cap R_1=Y\cap R_2=1$ because $Y$ is simple and transitive. Accordingly, there exists embeddings $\varphi_1,\varphi_2$ of $Y$ into $R$ such that $Y=\{(y^{\varphi_1},y^{\varphi_2})\mid y\in Y\}$. As $Y\cong H\cong T$ is nonabelian simple and $R/T$ is solvable, one sees steadily that $\varphi_1,\varphi_2$ are in fact isomorphisms from $Y$ to $T$. Taking an arbitrary point $\delta\in\Delta$, since $Y$ is transitive on $\Delta^2$, we have
$$
T=(Y^{\varphi_1}\cap T_\delta)^{\varphi_1^{-1}\varphi_2}T_\delta=T_\delta^{\varphi_1^{-1}\varphi_2}T_\delta.
$$
Then applying~\cite[Theorem~1.1]{Baumeister} we obtain $T=\A_6$, whence $\A_6\times\A_6\lesssim G\lesssim\Sy_6\wr\Sy_2$. Moreover, $H\cong\A_6$ and $G=HK$ for some core-free subgroup $K$ of $G$. Searching in \magma~\cite{magma} for such quasiprimitive permutation groups $G$ shows that $K$ must be maximal in $G$ of index $36$ or $360$. As a consequence, $G$ is primitive on $\Omega$. If $K$ has index $36$ in $G$, then $G$ is a primitive permutation group of degree $36$ with socle $\A_6\times\A_6$, and consulting~\cite[Appdendix~B]{DM} we conclude that part~(c) of Theorem~\ref{thm2} holds. If $K$ has index $360$ in $G$, then $H$ is regular on $\Omega$, which leads to part~(b) of Theorem~\ref{thm2} by~\cite{LPS2010}.
\qed

\noindent\textsc{Acknowledgements.}
The author acknowledges the support of China Postdoctoral Science Foundation Grant.

\end{document}